\documentclass[11pt]{article}
\usepackage{amsfonts,amsthm, amsmath}
\usepackage{epsfig}
\usepackage{makeidx}
\usepackage{graphicx}
\usepackage{graphicx,epstopdf}
\DeclareGraphicsExtensions{.ps,.png,.jpg}

\makeindex
\usepackage{color}

\def\E{{\mathbb E}}

\newcommand{\ncom}{\newcommand}
\ncom{\ul}{\underline}
\ncom{\beq}{\begin{equation}}
\ncom{\eeq}{\end{equation}}
\ncom{\bea}{\begin{eqnarray*}}
\ncom{\eea}{\end{eqnarray*}}
\ncom{\beqa}{\begin{eqnarray}}
\ncom{\eeqa}{\end{eqnarray}}
\ncom{\nno}{\nonumber}
\ncom{\non}{\nonumber}
\ncom{\ds}{\displaystyle}
\ncom{\half}{\frac{1}{2}}
\ncom{\mbx}{\makebox{.25cm}}
\ncom{\hs}{\mbox{\hspace{.25cm}}}
\ncom{\rar}{\rightarrow}
\ncom{\Rar}{\Rightarrow}
\ncom{\noin}{\noindent}
\ncom{\bc}{\begin{center}}
\ncom{\ec}{\end{center}}
\ncom{\sz}{\scriptsize}
\ncom{\rf}{\ref}
\ncom{\s}{\sqrt{2}}
\ncom{\sgm}{\sigma}
\ncom{\Sgm}{\Sigma}
\ncom{\psgm}{\sigma^{\prime}}
\ncom{\dt}{\delta}
\ncom{\Dt}{\Delta}
\ncom{\lmd}{\lambda}
\ncom{\Lmd}{\Lambda}
\ncom{\Th}{\Theta}
\ncom{\e}{\eta}
\ncom{\eps}{\epsilon}
\ncom{\pcc}{\stackrel{P}{>}}
\ncom{\lp}{\stackrel{L_{p}}{>}}
\ncom{\dist}{{\rm\,dist}}
\ncom{\sspan}{{\rm\,span}}
\ncom{\re}{{\rm Re\,}}
\ncom{\im}{{\rm Im\,}}
\ncom{\sgn}{{\rm sgn\,}}
\ncom{\ba}{\begin{array}}
\ncom{\ea}{\end{array}}
\ncom{\hone}{\mbox{\hspace{1em}}}
\ncom{\htwo}{\mbox{\hspace{2em}}}
\ncom{\hthree}{\mbox{\hspace{3em}}}
\ncom{\hfour}{\mbox{\hspace{4em}}}
\ncom{\vone}{\vskip 2ex}
\ncom{\vtwo}{\vskip 4ex}
\ncom{\vonee}{\vskip 1.5ex}
\ncom{\vthree}{\vskip 6ex}
\ncom{\vfour}{\vspace*{8ex}}
\ncom{\norm}{\|\;\;\|}
\ncom{\integ}[4]{\int_{#1}^{#2}\,{#3}\,d{#4}}
\ncom{\vspan}[1]{{{\rm\,span}\{ #1 \}}}
\ncom{\dm}[1]{ {\displaystyle{#1} } }
\ncom{\ri}[1]{{#1} \index{#1}}

\newtheorem{theorem}{\bf Theorem}[section]
\newtheorem{remark}{\bf Remark}[section]
\newtheorem{proposition}{Proposition}[section]

\newtheoremstyle
    {remarkstyle}
    {}
    {11pt}
    {}
    {}
    {\bfseries}
    {:}
    {     }
    {\thmname{#1} \thmnumber{#2} }

\theoremstyle{remarkstyle}

\begin{document}
%\input{iitbcover3.tex}
%\pagenumbering{roman}
%\input{cover1.txt}
%\input{cert.tex}
%\newpage
%\input{ack.tex}
%\tableofcontents
%\newpage
%\pagenumbering{arabic}
%\newpage
%\tableofcontains

\newpage

\begin{center}
{\Large \bf Large deviations of time-averaged statistics for Gaussian processes}
\end{center}
\vone
\begin{center}
 {J. Gajda}$^{\textrm{a}}$, {A. Wy{\l}oma{\'n}ska}$^{\textrm{a}}$, {H. Kantz}$^{\textrm{b}}$, {A. V. Chechkin}$^{\textrm{c,d}}$ and 
{G. Sikora}$^{\textrm{a}}$
{\footnotesize{
		$$\begin{tabular}{l}
		\\
		$^{\textrm{a}}$ \emph{Faculty of Pure and Applied Mathematics, Hugo Steinhaus Center,}\\\emph{Wroc{\l}aw University of Science and  Technology,  Wroc{\l}aw, Poland}\\
				$^{\textrm{b}}$ \emph{Max Planck Institute for the Physics of Complex Systems, Dresden, Germany}\\
				$^{\textrm{c}}$ \emph{Institute for Physics \& Astronomy, University of Potsdam, Potsdam-Golm, Germany}\\
			$^{\textrm{d}}$ \emph{Akhiezer Institute for Theoretical Physics NSC "Kharkov Institute of Physics and Technology",}\\
						\emph{Kharkov, Ukraine}\\
		\end{tabular}$$} }
\end{center}

\vtwo
\begin{center}
\noindent{\bf Abstract}
\end{center}
In this paper we study the large deviations of time averaged mean square displacement (TAMSD) for Gaussian processes. The theory of large deviations is related to the exponential decay of probabilities of large fluctuations in random systems. From the mathematical point of view a given statistics satisfies the large deviation principle, if the probability that it belongs to a certain range decreases exponentially. The TAMSD is one of the main statistics used in the problem of anomalous diffusion detection.  Applying the theory of generalized chi-squared distribution and sub-gamma random variables we prove the upper bound for large deviations  of TAMSD for Gaussian processes. As a special case we consider fractional Brownian motion,  one of the most popular models of anomalous diffusion. Moreover, we derive the upper bound for large deviations of the estimator for the anomalous diffusion exponent.  \\
\vone \noindent{\it Key words:} Large deviation statistics; fractional Brownian motion; anomalous diffusion exponent; sub-gamma random variable.
\vtwo
%\vtwo
\setcounter{equation}{0}

\section{Introduction}
 The theory of large deviations is concerned with the asymptotic behavior of large fluctuations of stochastic processes. The mathematical theory of large deviations was introduced by Cram\'er \cite{Cra38},   developed further in series of papers by Donsker and Varadhan \cite{Don1, Don2, Don3,Don4,var}, see also the monographs by Freidlin and Wentzell \cite{Fre84} and by Feng and Kurtz \cite{feng}, and references therein.

The theory of large deviations finds important applications in information theory \cite{dembo} and risk management \cite{novak}. However, the very first result of large deviations was obtained by 
Boltzmann more than one hundred years ago \cite{Ell9}. Indeed, prominent applications of large deviations theory arise in thermodynamics and statistical mechanics which deal with many particle systems \cite{touchette}. 

The intuitive definition of the large deviation principle could be given as follows:  let $A_N$ be a random variable indexed by the integer $N$ and let $P(A_N\in B$) be the probability that $A_N$ takes on a value in a set $B$. We say that $A_N$ satisfies a large deviation principle with the rate function $I_B$ if the following holds \cite{touchette}:
\begin{eqnarray}
 P(A_N\in B) \approx e^{-N I_B}.
\end{eqnarray}
The exact definition operates with supremum and infinum limits of the above probability and the rate function \cite{feng}. However, sometimes it is difficult or even impossible at all to find the explicit formula for the rate function or the large deviation principle. In this case one may still be able to find an upper bound  for the probability $P(A_N\in B)$, that is the function $I_B(N)$ which satisfies the following:
\begin{eqnarray}\label{eq12}
P(A_N\in B)\leq e^{-I_B(N)}.
\end{eqnarray}
This is exactly the case  we consider in this paper.

Recently, the large deviations for a variety of random variables have been intensively analyzed for different stochastic processes in the series of papers, see e.g. \cite{sp1,spa2,spa3,sp5,spa1}. In this paper we analyze large deviations of time averaged mean square displacement (TAMSD) for Gaussian processes. The TAMSD is the most common statistical tool characterizing anomalous diffusion or anomalous transport phenomena.  During last decades, anomalous diffusion have been identified in a great variety of complex physical, chemical and biological systems, see e.g. \cite{bouch, sok2, an1, metz1} and references therein.
The TAMSD at the lag time $\tau$  for a random vector $(X(1),X(2),...X(N))$ of length $N$ is defined as follows \cite{an1,an10}:
\begin{equation}
	\label{eq:msd}
	M_N(\tau)=\frac{1}{N-\tau}\sum_{j=1}^{N-\tau}(X(j+\tau)-X(j))^2.
\end{equation}
One of the main properties of TAMSD is its scaling. More precisely, this statistics behaves like  power function $\tau^\beta$, where $\tau$ is the lag time. For normal diffusion or ordinary Brownian motion (BM) the scaling is linear, i.e. $\beta=1$, and anomalous diffusion is characterized by $\beta\neq 1$. When $\beta>1$ then the process under consideration is superdiffusive while for $\beta<1$ it is called subdiffusive.  

One of the most generic processes, which exhibits anomalous diffusion behavior is the fractional Brownian motion (FBM) \cite{fbm,fbm1}. The FBM $\{B_H(t)\}$ is defined as zero-mean Gaussian process whose autocovariance takes the form:
\begin{equation}\label{cov}
\E(B_H(t)B_H(s))=D(|t|^{2H}+|s|^{2H}-|t-s|^{2H}).
\end{equation}
The parameter $H\in (0,1)$ is called the Hurst index. The FBM can exhibit sub- and super diffusion for $H<1/2$ and $H>1/2$, respectively. The newest results of the large deviation theory  concern different statistics of fractional Brownian motion \cite{Has13,Mar13,Mee08,Mil06,Wan18}. In this paper as a particular example we analyze large deviations of the TAMSD for FBM.

The rest of the paper is organized as follows. In Section  2 we remind the basic probability properties of TAMSD for Gaussian processes.  In Section 3 we present our main results about large deviations of TAMSD for Gaussian processes. Section 4 is devoted to a particular case, namely FBM.  In Section 5 we extend results on TAMSD and establish the large deviation principle for anomalous diffusion exponent estimator.

\setcounter{equation}{0}

\section{TAMSD for Gaussian processes}\label{section2}

In this section for the reader convenience we remind the probability properties of TAMSD for Gaussian processes, which were obtained recently \cite{greb11_1,greb11,bib:Sikora2017}. This will be the starting point for the main results of the paper, namely large deviation principle for TAMSD for Gaussian processes. 
 
For any centered non-degenerate Gaussian vector $\mathbb{X}=(X(1),X(2),\ldots,X(N))$ with a covariance matrix $\Sigma$ the quadratic form $\mathbb{X}\mathbb{X}^T$ has a following representation \cite{dav80,MatPro92}:
\begin{equation}\label{gchi}
\mathbb{X}\mathbb{X}^T=\sum_{j=1}^NX^2(j)\stackrel{d}{=}\sum_{j=1}^N\lambda_jU_j,
\end{equation}
where $\mathbb{X}^T$ is a transpose of vector $\mathbb{X}$ (so a column vector), $U_j$'s are independent identically distributed (IID) $\chi^2$ with $1$ degree of freedom 
random variables, and weights $\lambda_j$ are the eigenvalues of
the $N\times N$ positive-definite covariance matrix $\Sigma.$ The distribution in representation \eqref{gchi} is called a generalized chi-squared distribution.
Therefore for any centered non-degenerate Gaussian process, the TAMSD (\ref{eq:msd}) has
generalized chi-squared distribution
\cite{greb11}:
\begin{equation}\label{eq1}
(N-\tau)M_N(\tau)=\mathbb{Y}\mathbb{Y}^T=\sum_{j=1}^{N-\tau}\left(X(j+\tau)-X(j)\right)^2\stackrel{d}{=}\sum_{j=1}^{N-\tau}\lambda_j(\tau)U_j,
\end{equation}
where $U_j$'s are IID $\chi^2$ with $1$ degree of freedom 
random variables, and weights $\lambda_j(\tau)$ are the eigenvalues of
the $(N-\tau)\times(N-\tau)$ positive-definite covariance matrix ${\Sigma}(\tau)$ for
the vector of increments $\mathbb{Y}=(X(1+\tau)-X(1),X(2+\tau)-X(2),\ldots,X(N)-X(N-\tau))$. The covariance matrix takes the form: 
\begin{displaymath}
\Sigma(\tau) =
\begin{bmatrix}
  \sigma_{\tau}(0) & \ \ \ \sigma_{\tau}(1)\ \ \ &\ \ \ \sigma_{\tau}(2) \ \ \ & \ \ \ldots\ \  &\ \ \ \ldots \ \ \ &\sigma_{\tau}(N-\tau-1)  \\
  \sigma_{\tau}(1) & \sigma_{\tau}(0)  & \sigma_{\tau}(1) &  \ddots   &  &  \vdots \\
  \sigma_{\tau}(2)     & \sigma_{\tau}(1) & \ddots  & \ddots & \ddots& \vdots \\ 
 \vdots &  \ddots & \ddots &   \ddots  & \sigma_{\tau}(1) & \sigma_{\tau}(2) \\
 \vdots &         & \ddots & \sigma_{\tau}(1) & \sigma_{\tau}(0)&  \sigma_{\tau}(1) \\
\sigma_{\tau}(N-\tau-1)  & \ldots & \ldots & \sigma_{\tau}(2)  & \sigma_{\tau}(1) & \sigma_{\tau}(0)
\end{bmatrix}.
\end{displaymath}
All eigenvalues of matrix ${\Sigma}(\tau)$ are positive, $\lambda_j(\tau)>0,$ $j=1,2,\ldots,N-\tau.$
%It is known that for vector of observations corresponding to
%For stationary increments, the covariance matrix ${\Sigma}(\tau)$ is
%the Toeplitz matrix (diagonal-constant matrix)
%
%\begin{equation}
%\Sigma(\tau) = \left(\begin{array}{c c c c}
% \sigma_\tau(0) & \sigma_\tau(1) & \sigma_\tau(2) & \ldots \\
% \sigma_\tau(1) & \sigma_\tau(0) & \sigma_\tau(1) & \ldots \\
% \sigma_\tau(2) & \sigma_\tau(1) & \sigma_\tau(0) & \ldots \\
% \ldots   &  \ldots  &  \ldots  & \ldots \\
%\end{array}\right).
%\end{equation}
The distribution of the quadratic form $(N-\tau)M_N(\tau)$ in
Eq. (\ref{eq1}) can be represented as a sum of independent gamma
distributions with constant shape parameter $1/2$ and different scale
parameters, because
$\lambda_j(\tau)U_j\stackrel{d}{=}G(1/2,2\lambda_j(\tau))$ \cite{mos85}, where $G(k,\theta)$ is a gamma distributed random variable with parameters $k$ and  $\theta$.  We remind
that the probability density function (PDF) of $G(k,\theta)$ reads:
\begin{displaymath}
%\label{gam_pdf}
f_{(k,\theta)}(x)=\frac{x^{k-1} \,  \exp(-x/\theta)}{\Gamma(k) \,\theta^k} \quad (x>0),
\end{displaymath}
and the cumulative distribution function (CDF) of  $G(k,\theta)$ has the form:
$$F_{(k,\theta)}(x)=\frac{1}{\Gamma(k)}\gamma(k,x/\theta),$$
where $\gamma(k,x/\theta)$ is the incomplete lower gamma function of the general form:
$$\gamma(s,x)=\int_{0}^{x}{t^{s-1}e^{-t}dt.}$$
Therefore, the characteristic function of $(N-\tau)M_N(\tau)$ is the
product of characteristic functions of gamma distributions \cite{greb11_1,greb11}:
\begin{displaymath}
%\label{eq11}
\phi_{(N-\tau)M_N(\tau)}(k)=\prod_{j=1}^{N-\tau}\frac{1}{\left[1-2\lambda_j(\tau)ik\right]^{1/2}}.
\end{displaymath}
Based on the result of \cite{mos85} the moment generating function of $(N-\tau)M_N(\tau)$ can be represented as \cite{bib:Sikora2017}:
\begin{equation*}
\begin{split}
{\rm MGF}_{(N-\tau)M_N(\tau)}(s) &= C \bigl(1-2\lambda_1(\tau)s \bigr)^{-(N-\tau)/2} \exp\left(\sum_{k=1}^{\infty} \frac{\gamma_k}{\bigl(1-2\lambda_1(\tau)s \bigr)^k}\right), \\
\end{split}
\end{equation*}
where $\lambda_1(\tau)$ is the smallest eigenvalue of the matrix
$\Sigma(\tau)$,
%
%\begin{equation*}
%\lambda_{min}=\min\{\lambda_1(\tau),\ldots,\lambda_{N-\tau}(\tau)\}
%\end{equation*}
\begin{equation*}
\gamma_k=\sum_{j=1}^{N-\tau}\frac{(1-\lambda_1(\tau)/\lambda_{j}(\tau))^k}{2k}
\end{equation*}
and
\begin{equation*}
C=\prod_{j=1}^{N-\tau}\left(\frac{\lambda_1(\tau)}{\lambda_j(\tau)}\right)^{1/2}.
\end{equation*}
The PDF of $M_N(\tau)$ can be represented as a series of densities of gamma distributed variables $G((N-\tau)/2+k,2\lambda_1(\tau)/(N-\tau))$ \cite{bib:Sikora2017}:
\begin{eqnarray}\label{pdf}
g_{\tau}(x)=C\sum_{k=0}^{\infty}\delta_{k}f_{\left(\frac{N-\tau}{2}+k,\frac{2\lambda_1(\tau)}{(N-\tau)}\right)}(x),\ (x>0),
\end{eqnarray}
where $\delta_k$ can be calculated by the recursive formula:
\begin{equation*}
\delta_{k+1}=\frac{1}{k+1}\sum_{j=1}^{k+1} j  \gamma_j \delta_{k+1-j}, \quad \delta_0=1.
\end{equation*}
\noindent Since it is straightforward to see that $C\sum_{k=0}^{\infty}{\delta_{k}}=1$ and $C\delta_{k}>0$, the distribution of $M_N(\tau)$ with the PDF (\ref{pdf}) can be understood as a countable infinite mixture of gamma distributed random variables $G((N-\tau)/2+k,2\lambda_1(\tau)/(N-\tau))$.
The sequence in the above formula converges uniformly, see
\cite{mos85}. Therefore, justified term-by-term integration leads to the formula for the cumulative distribution function:
\begin{eqnarray}\label{eq11}
G_{\tau}(w)&=&P\left(M_{N}(\tau)\leq w\right)\nonumber\\
&=&C\sum_{k=0}^{\infty}\delta_k\int_{0}^{w}{f_{\left(\frac{N-\tau}{2}+k,\frac{2\lambda_1(\tau)}{(N-\tau)}\right)}(x)}dx\nonumber\\
&=&C\sum_{k=0}^{\infty}{\delta_kF_{\left(\frac{N-\tau}{2}+k,\frac{2\lambda_1(\tau)}{(N-\tau)}\right)}(w).}
\end{eqnarray}
From Eq. (\ref{eq11}) one can calculate also the formula for the complementary CDF (called also tail distribution) of $M_N(\tau)$:
\begin{eqnarray}
P\left(M_{N}(\tau)> w\right)&=&1-G_{\tau}(w)\nonumber\\
&=&1-C\sum_{k=0}^{\infty}{\delta_kF_{\left(\frac{N-\tau}{2}+k,\frac{2\lambda_1(\tau)}{(N-\tau)}\right)}(w).}
\end{eqnarray}
Therefore we obtain:
\begin{eqnarray*}
P\left(M_{N}(\tau)> w\right)=
C\sum_{k=0}^{\infty}\delta_k\Gamma\left(\frac{N-\tau}{2}+k,\frac{w(N-\tau)}{2\lambda_1(\tau)}\right),
\end{eqnarray*}
where $\Gamma(s,x)$ is incomplete upper  gamma function defined as:
$$\Gamma(s,x)=\int_{x}^{\infty}{t^{s-1}e^{-t}dt.}$$

\section{Large deviations of TAMSD for Gaussian processes}

Let us consider the following probability function:
\begin{eqnarray}\label{f1}
T_N(\epsilon,\tau)=P\left(\left|M_N(\tau)-\E(M_N(\tau))\right|>\epsilon\right).
\end{eqnarray}
In this section we check if the function (\ref{f1}) fulfills the large deviation principle in the sense of Eq. (\ref{eq12}), i.e. if there exists positive function $I(\epsilon,\tau,N)$ such that the following holds:
\begin{eqnarray}\label{f3}
P\left(\left|M_N(\tau)-\E(M_N(\tau))\right|>\epsilon\right)\leq e^{-I(\epsilon,\tau,N)}.
\end{eqnarray}
At first we observe that the random variable $(N-\tau)M_N(\tau)$ can be represented as:
\begin{eqnarray}\label{f11}
(N-\tau)M_N(\tau)=\sum_{j=1}^{N-\tau}G_j,
\end{eqnarray}
where $\{G_j\}$ are independent gamma random variables with parameters $(1/2,2\lambda_j(\tau))$, see section \ref{section2}. Now, let us remind the property of gamma distributed variables which will be of importance in what follows. Namely, the centered version of gamma distributed random variable $G$ with parameter $k$ and $\theta$: 
\begin{eqnarray}\label{G}\tilde{G} = G-\E\left( G \right)\end{eqnarray} is a sub-gamma variable with parameters $\nu=k\theta^2$ (variance factor) and $c=\theta$ (scale factor) \cite{bib:Boucheron}. 
 A random variable $X$ is called sub-gamma on the right tail with parameters $\nu$ and $c$, if the logarithm of its moment generating function satisfies \cite{bib:Boucheron}:
\begin{equation}\label{f2}
\log\left( \E\left( e^{\gamma X} \right)\right)  = \phi_X(\gamma)\leq \frac{\gamma^2 \nu}{2(1-c\gamma)}
\end{equation} 
for every $\gamma \in (0,1/c)$. Similarly, $X$ is sub-gamma on the left tail when $-X$ is sub-gamma on the right tail. For every sub-gamma variable with $\nu$ and $c$ parameters we have the following inequality (Chernoff's inequality):
\begin{align}\label{f33}
P\left( X>\epsilon \right) &\leq \exp\left( -\frac{\nu}{c^2} \mathcal{H}\left(  \frac{c\epsilon}{\nu}  \right)\right), 
\end{align}where
\begin{eqnarray}\label{H}\mathcal{H}(u)=1+u-\sqrt{1+2u}
\end{eqnarray} for $u>0$.

\noindent Going back to Eq. (\ref{f11}) one can see that
the inequality (\ref{f33}) is satisfied for each 
\begin{eqnarray}\label{Gj}\tilde{G}_j=G_j-\E(G_j)\end{eqnarray} with $\nu = 2\lambda^2_j(\tau)$  and $c = 2\lambda_j(\tau)$,  thus we have \cite{bib:Boucheron}:
{
\begin{align}\label{f3}
P\left( \tilde{G}_j>\epsilon \right) &\leq ~exp\left( -\frac{1}{2}\mathcal{H}\left( \frac{\epsilon}{\lambda^2_j(\tau)} \right)  \right).
\end{align}

\noindent In the following Proposition we show that also $(N-\tau)M_N(\tau)-\E\left( (N-\tau)M_N(\tau) \right)$ is a sub-gamma
random variable if the considered random vector $(X(1),X(2),..,X(N))$ is Gaussian.
\begin{proposition}
	\label{prop:1}
	We consider the vector  $(X(1),X(2),..,X(N))$ of centered non-degenerate Gaussian distributed random variables and its TAMSD, $M_N(\tau)$, given in Eq. (\ref{eq:msd}). Then for each $\tau$ the random variable: $$(N-\tau)M_N(\tau)-\E\left( (N-\tau)M_N(\tau) \right)$$ is a sub-gamma on the right tail.
\end{proposition}
\begin{proof} 
Let us observe that for fixed $\tau$ we have:
\begin{align*}
(N-\tau)M_N(\tau)-\E\left( (N-\tau)M_N(\tau) \right)=\sum\limits_{j=1}^{N-\tau}\tilde{G}_j,
\end{align*}
where $\tilde{G}_j$ is given in (\ref{Gj}). Moreover, $\{\tilde{G}_j\}$ are independent. Thus we obtain:
	\begin{align*}
	\log\left( \E\left( e^{\gamma \left((N-\tau)M_N(\tau)-\E\left( (N-\tau)M_N(\tau) \right)\right) }\right)\right)&=\log\left( \E\left( e^{\gamma \sum\limits_{j=1}^{N-\tau}\tilde{G}_j  }\right)\right) \\
	=& \log\left(\prod\limits_{j=1}^{N-\tau} \exp\left( -\gamma \lambda_j(\tau) -\frac{1}{2}\log\left( 1-2\gamma\lambda_j(\tau) \right) \right) \right)&\\
	&\leq  \log\left(\prod\limits_{j=1}^{N-\tau}  \exp\left( \frac{ 2 \gamma^2 \lambda^2_j(\tau) }{2\left( 1-2\gamma\lambda_j(\tau) \right)} \right)  \right)\\
	& \leq  \log\left(\prod\limits_{j=1}^{N-\tau}  \exp\left( \frac{2\gamma^2\lambda^2_j(\tau)}{2\left(  1-\gamma\bar{\lambda}(\tau) \right)}  \right)  \right)  \\
	& \leq   \frac{\gamma^2 \sum\limits_{j=1}^{N-\tau } 2\lambda^2_j(\tau)}{2(1-\gamma \bar{\lambda}(\tau))}.
	\end{align*}
	The first inequality is a consequence of the following property, namely for $u\in (0,1)$ we have:
	\begin{equation*}
	-\log(1-u) -u\leq \frac{u^2}{2(1-u)}. 
	\end{equation*}
	The above results show that  $(N-\tau)M_N(\tau)-\E\left( (N-\tau)M_N(\tau) \right)$ is sub-gamma random variable with 
	$\nu = 2\sum\limits_{j=1}^{N-\tau}\lambda^2_j(\tau)$  and $c = \bar{\lambda}(\tau) = 2 max_{j}\left\{  \lambda_j(\tau) \right\}$. It is worth mentioning that $\bar{\lambda}(\tau)$ may depend on $N$.
\end{proof}
\noindent Applying Proposition \ref{prop:1} one can provide the upper bound  for $T_N(\epsilon,\tau)$ defined in Eq. (\ref{f1}) for  Gaussian vector  $(X(1),X(2),...,X(N))$. 
\begin{theorem}
	\label{thm:2}
We consider the vector  $(X(1),X(2),..,X(N))$ of centered non-degenerate Gaussian distributed random variables and $M_N(\tau)$ is the TAMSD calculated for this vector according to formula (\ref{eq:msd}). For each $\tau$ the following inequality holds:
	\begin{align}
	\label{UpperBoundLD}
	P&\left(\left|M_N(\tau)-\E(M_N(\tau))\right|>\epsilon\right)
	&\leq 2 \exp\left(- \frac{2\sum\limits_{j=1}^{N-\tau}\lambda^2_j(\tau)}{\bar{\lambda}(\tau)^2}   \mathcal{H}\left( \frac{\bar{\lambda}(\tau) \epsilon (N-\tau)   }{2\sum\limits_{j=1}^{N-\tau}\lambda^2_j(\tau)}\right) \right),
	\end{align}
where $\mathcal{H}(\cdot)$ is given in (\ref{H}) and $\bar{\lambda}(\tau) = 2 max_{j}\left\{  \lambda_j(\tau) \right\}$.
\end{theorem}
\begin{proof}
	The proof of this Proposition follows directly from Proposition \ref{prop:1}.
	Since the random variable $(N-\tau)M_N(\tau)-\E\left( (N-\tau)M_N(\tau) \right)$ is a sub-gamma on the right tail, then  $-((N-\tau)M_N(\tau)-\E\left( (N-\tau)M_N(\tau) \right))$ is also sub-gamma on the left tail. Now, from Theorem 2.3 in \cite{bib:Boucheron} we get the inequality (\ref{UpperBoundLD}) which is a first main result of the present paper.
\end{proof}
}
%Using the above formula one can show:
%\begin{eqnarray*}
%P\left(\left|M_N(\tau)-E(M_N(\tau))\right|>\epsilon\right)=\\1-C\sum_{k=0}^{\infty}\frac{\delta_k}{\Gamma\left(\frac{N-\tau}{2}+k\right)}\gamma\left(\frac{N-\tau}{2}+k,\frac{(E(M_N(\tau))+\epsilon)(N-\tau)}{2\lambda_1(\tau)}\right)\\
%+C\sum_{k=0}^{\infty}\frac{\delta_k}{\Gamma\left(\frac{N-\tau}{2}+k\right)}\gamma\left(\frac{N-\tau}{2}+k,\frac{(E(M_N(\tau))-\epsilon)(N-\tau)}{2\lambda_1(\tau)}\right)\\
%=1-C\sum_{k=0}^{\infty}\frac{\delta_k}{\Gamma\left(\frac{N-\tau}{2}+k\right)}\int_{\frac{E(M_N(\tau))-\epsilon)(N-\tau)}{2\lambda_1(\tau)}}^{\frac{E(M_N(\tau))+\epsilon)(N-\tau)}{2\lambda_1(\tau)}}t^{(N-\tau)/2+k}e^{-t}dt.
%\end{eqnarray*}
%The above gives:
%\begin{eqnarray*}
%1-D_1\leq P\left(\left|M_N(\tau)-E(M_N(\tau))\right|>\epsilon\right)\leq 1-D_2,
%\end{eqnarray*}
%where for small $\epsilon$ $D_1$ and $D_2$ have the following form:
%\begin{eqnarray*}
%D_1=e^{-\frac{(E(M_N(\tau))-\epsilon)(N-\tau)}{2\lambda_1(\tau)}}\left(\frac{2C\epsilon}{E(M_N(\tau))+\epsilon}\right)\sum_{k=0}^{\infty}\frac{\delta_k}{\Gamma\left(\frac{N-\tau}{2}+k\right)}
%\end{eqnarray*}
%\begin{eqnarray*}
%D_2=e^{-\frac{(E(M_N(\tau))+\epsilon)(N-\tau)}{2\lambda_1(\tau)}}\left(\frac{2C\epsilon}{E(M_N(\tau))+\epsilon}\right)\sum_{k=0}^{\infty}\frac{\delta_k}{\Gamma\left(\frac{N-\tau}{2}+k\right)}
%\end{eqnarray*}
\begin{remark}\label{rem1}
Form Perron-Frobenius theorem \cite{Mey00} it is known that the maximum eigenvalue $\max_{j}{\lambda_j(\tau)}$ of the matrix $\Sigma(\tau)$ exists and satisfies the following inequalities:
$$\textrm{minimum raw sum}\leq max_{j}{\lambda_j(\tau)}\leq\textrm{maximum raw sum}.$$
For matrix $\Sigma(\tau),$ in the case of $N-\tau$ as an odd number, the minimum raw sum is the sum of elements of first raw (also $(N-\tau)$-th raw) and the maximum raw sum is the sum of elements of $((N-\tau-1)/2+1)$-th raw. Therefore we have the following:
\begin{equation}\label{eq1}
\sum_{j=0}^{N-\tau-1}\sigma_{\tau}(j)\leq max_{j}{\lambda_j(\tau)}\leq\sigma_{\tau}(0)+2\sum_{j=1}^{(N-\tau-1)/2}\sigma_{\tau}(j).
\end{equation}
For matrix $\Sigma(\tau),$ in the case of $N-\tau$ as an even number, the minimum raw sum is the sum of elements of first raw (also $(N-\tau)$-th raw) and the maximum raw sum is the sum of elements of $((N-\tau)/2)$-th raw (also $((N-\tau)/2+1)$-th raw). Therefore we have the following:
\begin{equation}\label{eq2}
\sum_{j=0}^{N-\tau-1}\sigma_{\tau}(j)\leq max_{j}{\lambda_j(\tau)}\leq\sigma_{\tau}(0)+2\sum_{j=1}^{(N-\tau)/2-1}\sigma_{\tau}(j).
\end{equation}

\end{remark}

\setcounter{equation}{0}
\section{Large deviation principle for TAMSD of fractional Brownian motion}\label{TAMSD_fbm}
In this section we  present upper bounds for the probability function  $T_N(\epsilon,\tau)$ given in (\ref{f1}),  for BM and FBM.
\begin{proposition}\label{tpro41}
Let $(B(1),B(2),...,B(N))$ be a random vector of BM and $M_N(\tau)$ - the corresponding TAMSD. In this case the following inequality holds:
\begin{eqnarray}\label{prop11}
	P\left(\left|M_N(\tau)-\E(M_N(\tau))\right|>\epsilon\right) 
	 \nonumber \\ \leq2 \exp\left(- \frac{4(N-\tau)D^2\tau(\tau+1)(2\tau+1)}{3\bar{\lambda}(\tau)^2}   \mathcal{H}\left( \frac{3\bar{\lambda}(\tau) \epsilon  }{4D^2\tau(\tau+1)(2\tau+1)}\right) \right)
\end{eqnarray}
where $\mathcal{H}(\cdot)$ is defined in (\ref{H}) and $\bar{\lambda}(\tau) = 2 max_{j}\left\{  \lambda_j(\tau) \right\}$.
\end{proposition}
\normalsize
\begin{proof}
It is enough to find the formula for the sum of squared eigenvalues $\sum\limits_{j=1}^{N-\tau}\lambda^2_j(\tau)$ and put it into inequality \eqref{UpperBoundLD}. We apply two well-known facts from linear algebra. The first fact states that the sum of all eigenvalues of some matrix is identified to the trace of that matrix. The second fact is that  the squared eigenvalues are the eigenvalues of the matrix to the second power. Therefore, we have for eigenvalues $\lambda_j(\tau)$ and the covariance matrix $\Sigma(\tau)$ \cite{gant}:
$$trace(\Sigma^2(\tau))=\sum\limits_{j=1}^{N-\tau}\lambda^2_j(\tau).$$ The matrix $\Sigma^2(\tau)$ has a main diagonal with the same element $\sum_{j=0}^{N-\tau-1}\sigma_{\tau}^2(j)$, where 
%\begin{displaymath}
\begin{eqnarray}\label{sigma_mala}
\sigma_\tau(j)=\left\{ \begin{array}{ll}
2D(\tau-j) & \textrm{$j\leq\tau-1$}\\
0 & \textrm{$j>\tau-1$}\\
\end{array} \right..
\end{eqnarray}
%\end{displaymath}
Hence we have:
\begin{align}\label{pro41}\sum\limits_{j=1}^{N-\tau}\lambda^2_j(\tau)&=(N-\tau)\sum_{j=0}^{N-\tau-1}\sigma_{\tau}^2(j)\nonumber\\
&=(N-\tau)\sum_{j=0}^{\tau-1}[2D(\tau-j)]^2=(N-\tau)4D^2\left[\frac{\tau(\tau+1)(2\tau+1)}{6}\right],
\end{align}
where the last equality is based on the known formula: $$\sum_{k=1}^nk^2=n(n+1)(2n+1)/6.$$ 
Substituting \eqref{pro41} into inequality \eqref{UpperBoundLD} we arrive at (\ref{prop11}).

	\end{proof}
\begin{remark}
In order to obtain the asymptotics for $\bar{\lambda}(\tau)$ in Eq. (\ref{prop11}) we use the Remark \ref{rem1}. Indeed, taking under consideration Eq. (\ref{sigma_mala}) for BM we obtain:
$$\sum_{j=0}^{N-\tau-1}\sigma_{\tau}(j)=\sum_{j=0}^{\tau-1}2D(\tau-j)=D\tau(\tau+1),$$
where the last equality based on formula:
\begin{equation}\label{eq3}
\sum_{k=1}^nk=\frac{n(n+1)}{2}.
\end{equation}
Moreover we have:
$$\sigma_{\tau}(0)+2\sum_{j=1}^{(N-\tau-1)/2}\sigma_{\tau}(j)=\sigma_{\tau}(0)+2\sum_{j=1}^{(N-\tau)/2-1}\sigma_{\tau}(j)=2D\tau+2\sum_{j=1}^{\tau-1}2D(\tau-j)=2D\tau^2,$$
where the last equality is based on \eqref{eq3}.
Hence the inequalities \eqref{eq1} and \eqref{eq2} for BM reduce to the following:
$$D\tau(\tau+1)\leq max_j{\lambda_j(\tau)}\leq2D\tau^2.$$
\end{remark}
{
\begin{proposition}\label{tpropo5}
Let $(B_H(1),B_H(2),...,B_H(N))$ be a random vector of FBM and $M_N(\tau)$ - the corresponding TAMSD. In this case the following inequality holds:
\begin{eqnarray}\label{prop55}
	P\left(\left|M_N(\tau)-\E(M_N(\tau))\right|>\epsilon\right) 
	\leq 2 \exp\left(- \frac{2(N-\tau)D^2\alpha(\tau,H,N)}{\bar{\lambda}(\tau)^2}   \mathcal{H}\left( \frac{\bar{\lambda}(\tau) \epsilon }{2D^2\alpha(\tau,H,N)}\right) \right),
\end{eqnarray}
where $\mathcal{H}(\cdot)$ is defined in (\ref{H}), $\bar{\lambda}(\tau) = 2 max_{j}\left\{  \lambda_j(\tau) \right\}$, and
\begin{eqnarray}\label{alfa}
	\alpha(\tau,H,N)=\sum_{i=0}^{N-\tau-1}\left[(i+\tau)^{2H}-2i^{2H}+|i-\tau|^{2H}\right]^2.
	\end{eqnarray}
\end{proposition}
\begin{proof}
\normalsize
The proof is analogous to that of Proposition \ref{tpro41}. It is enough to find the formula for sum of squared eigenvalues $\sum\limits_{j=1}^{N-\tau}\lambda^2_j(\tau),$ where  $\lambda^2_j(\tau)$ are the eigenvalues of the matrix $\Sigma^2(\tau)$, where $\Sigma(\tau)$ is a covariance matrix of the FBM increments. 
Therefore, we have:
$$trace(\Sigma^2(\tau))=\sum\limits_{j=1}^{N-\tau}\lambda^2_j(\tau).$$ The matrix $\Sigma^2(\tau)$ has  main diagonal with the same element $\sum_{i=0}^{N-\tau-1}\sigma_{\tau}^2(i).$ Hence, finally for FBM with: 
\begin{eqnarray}\label{sigma_mala_fbm}
\sigma_{\tau}(i)=D\left[(i+\tau)^{2H}-2i^{2H}+|i-\tau|^{2H}\right]\end{eqnarray} we get:
\begin{eqnarray}\label{eqr1}\sum\limits_{j=1}^{N-\tau}\lambda^2_j(\tau)=(N-\tau)\sum_{i=0}^{N-\tau-1}\sigma_{\tau}^2(i)=(N-\tau)D^2\sum_{i=0}^{N-\tau-1}\left[(i+\tau)^{2H}-2i^{2H}+|i-\tau|^{2H}\right]^2.\end{eqnarray}
Taking above to the inequality (\ref{UpperBoundLD}) we obtain the result (\ref{prop55}).
\end{proof}
}
\begin{remark}Taking into account Eq. (\ref{sigma_mala_fbm}) we can obtain the inequalities in \eqref{eq1} for FBM in the case $N-\tau$ is odd:
\small
\begin{equation}\label{eq4}
\sum_{j=0}^{N-\tau-1}\!\!\!D\left[(j+\tau)^{2H}-2j^{2H}+|j-\tau|^{2H}\right]\leq\max_{j}{\lambda_j(\tau)}\leq\sigma_{\tau}(0)+2\!\!\!\sum_{j=1}^{(N-\tau-1)/2}\!\!\!D\left[(j+\tau)^{2H}-2j^{2H}+|j-\tau|^{2H}\right].
\end{equation}
\normalsize
In the case $N-\tau$ is even the inequalities \eqref{eq2} have the form:
\small
\begin{equation}\label{eq5}
\sum_{j=0}^{N-\tau-1}\!\!\!D\left[(j+\tau)^{2H}-2j^{2H}+|j-\tau|^{2H}\right]\leq\max_{j}{\lambda_j(\tau)}\leq\sigma_{\tau}(0)+2\!\!\!\sum_{j=1}^{(N-\tau)/2-1}\!\!\!D\left[(j+\tau)^{2H}-2j^{2H}+|j-\tau|^{2H}\right].
\end{equation}
\end{remark}
\begin{remark}We can simplify results \eqref{eq4} and \eqref{eq5} for FBM in the case $\tau=1$. Based on the formula:
$$\sum_{j=1}^n\left[(j+1)^{2H}-2j^{2H}+(j-1)^{2H}\right]=(n+1)^{2H}-n^{2H}-1,$$
we can obtain the simplified version of \eqref{eq4}, namely for FBM in the case $N$ is even and $\tau=1$ we have:
\begin{equation}\nonumber
D\left[(N-1)^{2H}-(N-2)^{2H}+1\right]\leq max_{j}{\lambda_j(\tau)}\leq2D\left[\left(\frac{N}{2}\right)^{2H}-\left(\frac{N-2}{2}\right)^{2H}\right].
\end{equation}
In the case $N$ is odd and $\tau=1$ for FBM we obtain (the simplified version of \eqref{eq5}):
\begin{equation}\nonumber
D\left[(N-1)^{2H}-(N-2)^{2H}+1\right]\leq max_{j}{\lambda_j(\tau)}\leq2D\left[\left(\frac{N-1}{2}\right)^{2H}-\left(\frac{N-3}{2}\right)^{2H}\right].
\end{equation}

\end{remark}
\section{Large deviation principle for estimator of  anomalous diffusion exponent}

\label{sec:5}
Let us extend the results of the previous sections to derive one-sided  large deviation principle of the estimator for anomalous diffusion exponent of FBM. To simplify the presentation of this section without  loss of generality we assume $D=1/2$, see Eq. (\ref{cov}). The generalization is straightforward. Due to the stationarity of the increments, for the long enough trajectory of FBM we have the following asymptotic equality: 
\begin{equation*}
M_N(\tau) \simeq  \mathbb{E}\left( M_N(\tau) \right) =\tau^\beta.
\end{equation*}
from which we can immediately deduce that the  estimator of the parameter $\beta$  has the following distribution:
\begin{equation*}
\hat{\beta} \,{\buildrel d \over =}\, \ln\left( M_N(\tau) \right)/\ln\left(\tau \right).
\end{equation*}

\begin{proposition}
	Let us consider zero-mean FBM whose autocovariance is given by Eq. (\ref{cov}) with $D=1/2$. For a large sample length $N$ the anomalous diffusion exponent estimator $\hat{\beta}$ satisfies the following inequality:
	\begin{align}\label{res:thm3}
	P\left(  \hat{\beta}- \mathbb{E}\left( \hat{\beta} \right)  > \varepsilon \right)\leq & \exp\left(- \frac{(N-\tau)\alpha(\tau,H,N)}{2\bar{\lambda}(\tau)^2}   \mathcal{H}\left( \frac{2\bar{\lambda}(\tau) \left(\tau^{\varepsilon +\beta} - \tau^\beta \right)   }{\alpha(\tau,H,N)}\right) \right),
	\end{align}
where $\bar{\lambda}(\tau) = 2 max_{j}\left\{  \lambda_j(\tau) \right\}$, $\mathcal{H}(\cdot)$ is defined in (\ref{H}), and $\alpha(\tau,H,N)$ is given in (\ref{alfa}).
	
\end{proposition}
\begin{proof}
	
	In \cite{bib:Sikora2017} it was proved the $\hat{\beta}$ is asymptotically unbiased, that is 
	\begin{eqnarray}
	\E(\hat{\beta}) \xrightarrow[]{N\rightarrow\infty}\beta .
	\end{eqnarray}
	Then in order to determine the one-sided {large deviation}  for $\hat{\beta}$ one has to establish:
	\begin{align*}
	P&\left(\ln\left( \frac{M_N(\tau)}{\tau^\beta} \right)> \varepsilon\ln\left( \tau \right) \right)=
	P\left( \ln\left( \frac{M_N(\tau)}{\tau^\beta} \right)  > \varepsilon\ln\left( \tau \right) \right)
		=P\left( M_N(\tau)  > \tau^{\varepsilon +\beta} \right).
		\end{align*}
Using Theorem \ref{thm:2} we get:
\begin{eqnarray}\label{eq58}
P\left( M_N(\tau)  > \tau^{\varepsilon +\beta} \right) &=
P\left( (N-\tau)\left( M_N(\tau)-\mathbb{E}\left( M_N(\tau)\right) \right)  > (N-\tau)\left(\tau^{\varepsilon +\beta} - \mathbb{E}\left( M_N(\tau)\right) \right) \right)\nonumber\\
&\leq \exp\left(- \frac{2\sum\limits_{j=1}^{N-\tau}\lambda^2_j(\tau)}{\bar{\lambda}(\tau)^2}   \mathcal{H}\left( \frac{\bar{\lambda}(\tau) \left(\tau^{\varepsilon +\beta} - \tau^\beta \right) (N-\tau)   }{2\sum\limits_{j=1}^{N-\tau}\lambda^2_j(\tau)}\right) \right).
\end{eqnarray}
Taking into account Eq. (\ref{eqr1}) and inequality (\ref{eq58}) we obtain inequality (\ref{res:thm3}), which is the second main result of the present paper.
\end{proof}

\section*{Acknowledgements}
GS would like to acknowledge a support of NCN Maestro Grant No.2012/06/A/ST1/00258. \\JG and AW would like to acknowledge a support of National Center of Science Opus Grant No. 2016/21/B/ST1/00929 "Anomalous diffusion processes and their applications in real data modelling".\\
AVC acknowledges funding from the Deutsche Forschungsgemeinschaft, project ME 1535/6-1.

\end{document}